\documentclass[12 pt,a4paper]{amsart} 
\usepackage{amsfonts,amssymb,amscd,amsmath,enumerate,verbatim,calc} 
\usepackage{float}
\usepackage{amsthm}
\newtheorem{theorem}{Theorem}[section]
\newtheorem{definition}{Definition}[section]

\newtheorem{lemma}[theorem]{Lemma}
\newtheorem{corollary}[theorem]{Corollary}
\newtheorem{proposition}[theorem]{Proposition}
\newtheorem{remark}[theorem]{Remark}
\usepackage{tikz-cd}
\usepackage{mathtools}
\usepackage{xfrac}
\usepackage[all,cmtip]{xy}
\newtheorem{example}[theorem]{Example}
\numberwithin{equation}{section}

\usepackage{amsfonts}
\newcommand{\field}[1]{\mathbb{#1}}

\newcommand{\C}{\field{C}}

\newcommand{\im}{\rm Im}

\renewcommand{\ker}{\textnormal{ker}}
\renewcommand{\im}{\textnormal{Im}}

\begin{document}
	
\title{The Bogomolov multiplier of a multiplicative Lie algebra}
\author{Amit Kumar$^{1}$, Renu Joshi$^{2}$, Mani Shankar Pandey$^{3}$ and Sumit Kumar Upadhyay$^{4}$\vspace{.4cm}\\
		{$^{1, 4}$Department of Applied Sciences,\\ Indian Institute of Information Technology Allahabad\\Prayagraj, U. P., India} \vspace{.3cm}\\ $^{2}$ Department of Mathematics, \\Indian Institute of Science Education and Research \\Bhopal, Madhya Pradesh\\ $^{3}$ Department of Sciences, \\Indian Institute of Information Technology \\Design and Manufacturing, Kurnool, Andhra Pradesh}
	
	\thanks{$^1$amitiiit007@gmail.com, $^2$renu16@iiserb.ac.in, $^3$manishankarpandey4@gmail.com,  $^4$upadhyaysumit365@gmail.com}

\begin{abstract}
In this paper, we develop  the concept of the Bogomolov multiplier for a multiplicative Lie algebra and establish a Hopf-type formula. Consequently, we see that the Bogomolov multipliers of two isoclinic multiplicative Lie algebras are isomorphic.
\end{abstract}
	\maketitle
	
\section{Introduction}
The Bogomolov multiplier  is a group theoretical invariant which was introduced as an  obstruction  to an important problem in invariant theory discussed by Emmy Noether \cite{Noether}. In 1988, Bogomolov \cite{Bogomolov88} showed that the unramified cohomology group is canonically isomorphic to a certain subgroup ${B_0}(G)$ of the Schur multiplier of a group $G$. Kunyavskii \cite{BK} termed $B_0(G)$ as the Bogomolov multiplier of $G$.  Moravec \cite{PM} provided a new description of the Bogomolov multiplier using the non-abelian exterior square $G\wedge G$ of a group $G$ by showing that $B_0(G)$  is isomorphic to Hom$(\Tilde{B}_0(G),\frac{\mathbb{Q}}{\mathbb{Z}}),$ where the group $\Tilde{B}_0(G)$ can be describe as a section of the non-abelian exterior square.

Following Moravec's construction, Rostami et.al \cite{ZMP} extended the notion of the Bogomolov multiplier to Lie algebras. In the work of Rostami et.al, the Bogomolov multiplier of a Lie algebra $L$ over a field $\Omega$ is defined as a particular factor of a subalgebra of the exterior product $L\wedge L$. P. K. Rai \cite{PK} identified this object as a certain subgroup of the second cohomology group $H^2(L,\Omega)$ by deriving a Hopf-type formula and answered the question affirmatively posed by Kunyavskii regarding the invariance of the Bogomolov multiplier under isoclinism of Lie algebras. 

A multiplicative Lie algebra \cite{GJ} is a group with an extra binary operation which satisfies a non-commutative version of the usual axioms of the Lie bracket.   In recent years, there has been a notable interest in advancing the theory of multiplicative Lie algebras. In 2019,  Lal and Upadhyay \cite{RLS} studied the structure of multiplicative Lie algebras,  the theory of extensions, the second cohomology groups of multiplicative Lie algebras, and in turn the Schur multipliers.  In this paper, we attempt to develop the concept of the Bogomolov multiplier for a multiplicative Lie algebra. We define the Bogomolov multiplier as a certain subgroup of the second cohomology group $H^2_{ML}(G,\C^*)$ of a multiplicative Lie algebra $G$. Also, we define the group 
$\Tilde{B}_0(G)$ as a section of the Lie exterior square and establish an isomorphism between the groups ${B_0}(G) $ and $ \Tilde B_0(G) $ for a finite multiplicative Lie algebra $G$. Consequently, we derive a Hopf-type formula and a five term exact sequence for the Bogomolov multiplier $\Tilde{B}_0(G)$. The Lie exterior square \cite{RLS} provides essential insights into the Schur multiplier of a multiplicative Lie algebra . So, we introduce a related construction, namely the curly Lie exterior square that plays a similar role when examining $\Tilde{B}_0(G)$. We also define the CTP extension to get a universal CTP extension of a perfect multiplicative Lie algebra $G$ by $\Tilde{B}_0(G)$. Recently, Pandey and Upadhyay \cite{MS3} has defined the notion of isoclinism in multiplicative Lie algebras. We show that the Bogomolov multipliers of isoclinic multiplicative Lie algebras are isomorphic.

\section{Preliminaries}
In this section, we revisit the fundamental concepts and properties associated with multiplicative Lie algebras and the Lie exterior square. For a more in-depth exploration of these topics, we recommend referring to \cite{RLS}. 
\begin{definition}\label{D1}\cite{GJ}
	A multiplicative Lie algebra is a triple $ (G,\cdot,\star), $ where $G$ is a set, $ \cdot $ and $ \star $ are two binary operations on $G$ such that $ (G,\cdot) $ is a group (need not be abelian) and for all  $x, y, z  \in  G$, the following identities hold: 
	\begin{enumerate}
		\item $ x\star x=1 $  
		\item $ x\star(yz)=(x\star y){^y(x\star z)} $ 
		\item $ (x 
		y)\star z= {^x(y\star z)} (x\star z) $ 
		\item $ ((x\star y)\star {^yz})((y\star z)\star{^zx})((z\star x)\star{^xy})=1 $ 
		\item $ ^z(x\star y)=(^zx\star {^zy})$ 
	\end{enumerate}
	where $^xy$ denotes $xyx^{-1}.$ We say that $ \star $ is a multiplicative Lie algebra structure on the group $G$.
\end{definition}

Here, we  recall the definition of  the second Lie cohomology group for a multiplicative Lie algebra $G$ as introduced in \cite{RLS}.
 
	A multiplicative Lie 2-cocycle of a multiplicative Lie algebra $G$ with coefficients in an abelian group $H$ with trivial Lie product is a pair $(f,h)$, where $ f \in Z^2(G,H) $ is a group 2-cocycle of $G$ with coefficients in the trivial $G$-module $H$, and $h$ is a map from $ G \times G $ to $H$ satisfying the following equations for all $x, y, z \in G$:
	
	\begin{enumerate}
		\item $ h(x,1)=h(1,x)=h(x,x)=1 $ ,
		
		\item $ h(x,yz)=h(x,y)h(x,z)f(y^{-1},y)^{-1}f(y,x\star z) f(y(x\star z),y^{-1})f(x\star y,^y(x\star z)) $ ,
		
		\item $ h(xy,z)=h(y,z)h(x,z)f(x^{-1},x)^{-1}f(x,y\star z) f(x(y\star z),x^{-1})f(^x(y\star z),(x\star z)) $, 
		
		\item $ h(y\star x,^xz)h(x\star z,^zy)h(z\star y,^yx) f((y\star x)\star ^xz,(x\star z)\star ^zy)f((y\star x)\star ^xz)\cdot(x\star z)\star ^zy, \linebreak (z\star y)\star ^yx)=1 $, 
		
		\item $ h(^zx,^zy)=h(x,y)f(z,x\star y)f(z^{-1},z)^{-1}f(z(x\star y),z^{-1}) $. 
	\end{enumerate}
	
	The set $ Z^2_{ML}(G,H) $ of all multiplicative Lie 2-cocycles of $G$ with coefficients in $H$ is an abelian group with respect to the coordinate wise operation given by $$ (f,h)\cdot (f',h')=(ff',hh').$$
	Given any identity preserving map $g$ from $G$ to $H$, the pair $(\delta g,g^*) $ of maps from $ G\times G $ to H given by $$\delta g(x,y)=g(y)g(xy)^{-1}g(x)$$ and $$g^*(x,y)=g(x\star y)^{-1}$$ is a member of $ Z^2_{ML}(G,H). $
	Let $MAP(G,H)$ denote the group of identity preserving maps from $G$ to $H$. We have a group homomorphism $$\chi: MAP(G,H) \longrightarrow Z^2_{ML}(G,H) $$ given by  $$\chi(g)=(\delta g,g^*). $$
	The image of $\chi$ is called the group of multiplicative Lie 2-coboundaries of $G$ with coefficients in $H$ and is denoted by  $ B^2_{ML}(G,H). $ 
	The quotient group  $ \frac{Z^2_{ML}(G,H)}{ B^2_{ML}(G,H)}$ is called the second Lie cohomology of $G$ with coefficients in $H$, and is denoted by  $ H^2_{ML}(G,H). $

\begin{proposition}\cite[Corollary 6.3]{RLS}
	Let $G$ be a finite multiplicative Lie algebra having a free presentation
	\begin{align*}
		1\longrightarrow R \longrightarrow F \longrightarrow G\longrightarrow 1.
	\end{align*}
	Then $H^2_{ML}(G,\C^*)$ is isomorphic to $\frac{R\cap (F\star F)[F,F]}{(R\star F)[R,F]}$.
\end{proposition}

\begin{definition}\cite{RLS}
Let 
	\begin{align*}
	1\longrightarrow R \longrightarrow F \longrightarrow G\longrightarrow 1
\end{align*}
be a free presentation of the multiplicative Lie algebra $G$. Then $\tilde{M}(G) = \frac{R\cap (F\star F)[F,F]}{(R\star F)[R,F]}$ is called the Schur multiplier of the multiplicative Lie algebra $G$.
\end{definition}

\begin{remark}
For a finite multiplicative Lie algebra $G$ having a free presentation
\begin{align*}
	1\longrightarrow R \longrightarrow F \longrightarrow G\longrightarrow 1,
\end{align*}
we have $H^2_{ML}(G,\C^*)\cong \tilde{M}(G) $.
\end{remark}

\begin{definition}\cite{RLS}\label{Lie}
 Let $(G,\cdot, \star)$ be a multiplicative Lie algebra. Consider the multiplicative Lie algebra $G \wedge^L G$ generated by the set  $\{x\wedge y:x,y\in G\}\cup \{[x,y]_0:x,y\in G \}$ of formal symbols subject to the following universal relations: 

\begin{enumerate}
	\item $ 1 \wedge x = x\wedge 1 = x\wedge x = 1 = [1,x]_0 = [x,1]_0 = [x,x]_0$
	
	\item $ (x\wedge y)(y\wedge x) = 1 =[x,y]_0 [y,x]_0 $
	
	\item $(x\wedge yz)^{-1}(x\wedge y)(^yx\wedge ^yz)=1=[x,yz]_0^{-1}[x,y]_0[^yx,^yz]_0$
	
	\item $(xy\wedge z)^{-1}(^xy\wedge ^xz)(x\wedge z)=1=[xy,z]_0^{-1}[^xy,^xz]_0[x,z]_0$
	
	\item $[x\wedge y,u\wedge v]=[[x,y]_0,u\wedge v]=[x,y]_0 \tilde \star(u\wedge v)=[x,y]_0\tilde{\star}[u,v]_0$
	
	\item $(x\wedge y)(^zy\wedge {^zx})=(^{xy}x^{-1}\wedge {^xz})(x\wedge z)$
	
	\item $[x,y]_0[^zy,{^zx}]_0=[^{xy}x^{-1},{^xz}]_0[x,z]_0$
	
	\item $[[x,y]_0,[u,v]_0]=[[x,y],[u,v]]_0$
	
	\item $(x\wedge y)\tilde{\star}(u\wedge v)=(x\star y)\wedge (u\star v)$
	
\end{enumerate}	
where $  x,y,z,u,v\in G.$ The multiplicative Lie algebra $G \wedge^L G$  is said to be the\textit{ Lie exterior square} of $G$.
\end{definition}

From \cite{RLS}, there is a unique multiplicative Lie algebra homomorphism $\chi:G\wedge ^L G \to G$ given by $x\wedge y\mapsto x\star y$ and $[z,w]_0\mapsto [z,w]$. Thus, we have the short exact sequence
\begin{align*}\label{S1}
	1 \longrightarrow \ker(\chi) \longrightarrow G \wedge^L G  \overset \chi\longrightarrow (G\star G)[G,G] \longrightarrow 1
\end{align*}
of multiplicative Lie algebras.	

\begin{theorem}\cite{RLS}\label{T1}
	If $F$ is a free multiplicative Lie algebra, then $F\wedge ^L F\cong (F\star F)[F,F]$ and $\ker(\chi) = \{1\}.$ Further, if 
	\begin{align*}
		1\longrightarrow R \longrightarrow F \longrightarrow G\longrightarrow 1
	\end{align*}
	is a presentation of the multiplicative Lie algebra $G$,  then $\ker(\chi)\cong \frac{R\cap (F\star F)[F,F]}{(R\star F)[R,F]} \cong \tilde{M}(G)$ and $G \wedge^L G \cong \frac{(F\star F)[F,F]}{(R\star F)[R,F]}.$
\end{theorem} 

\begin{definition}
Two  multiplicative Lie algebras $G$ and $H$ are said  to be isoclinic (written as $G \sim_{ml} H$) if  there exist multiplicative Lie algebra isomorphisms $\lambda:\frac{G}{\mathcal{Z}(G)} \longrightarrow \frac{H}{\mathcal{Z}(H)}$ and $\mu:\ ^M[G, G]\longrightarrow \ ^M[H, H]$ such that the image of $^M[a,b]$ under $\mu$ is compatible with the image of $a\mathcal{Z}(G)$ and $b\mathcal{Z}(G)$ under $\lambda$. In other words the following diagram

\[
\xymatrix{
^M[G, G]\ar[d]^{\mu} & \frac{G}{\mathcal{Z}(G)}\times \frac{G}{\mathcal{Z}(G)}\ar[l]^{\phi_{c}~~~~~}  \ar[r]_{~~~~\phi_{s}} \ar[d]_{\lambda\times \lambda} &^M[G, G]  \ar[d]_{\mu} \\
^M[H, H] & \frac{H}{\mathcal{Z}(H)}\times \frac{H}{\mathcal{Z}(H)}\ar[l]_{\psi_{c}~~~~~~} \ar[r]^{~~~~\psi_{s}} & ^M[H, H]
}
\]

is commutative. The pair $(\lambda,\mu)$ is called as isoclinism between the  multiplicative Lie algebras $G$ and $H$.
\end{definition}

\section{Bogomolov multiplier}

In this section, we present the cohomological definition of the Bogomolov multiplier $B_0(G)$ for a multiplicative Lie algebra  $G$. Here we note that $G$ always present a finite multiplicative Lie algebra.  We establish an isomorphism between $B_0(G)$ and the group $\tilde B_0(G)$ when dealing with a finite multiplicative Lie algebra $G$. To achieve this, we commence by introducing a restriction map as follows:

Let $G$  be a multiplicative Lie algebra and $H$ a subalgebra of $G.$ Let $A$ be an abelian group with trivial Lie product. 
Given $(f,h)\in  Z^2_{ML}(G,A)$, let $f_H : H\times H \to A$ and $h_H : H\times H \to A$ be the restrictions of $f$ and $h$ to $ H\times H $ respectively.
Then $(f_H,h_H)\in   Z^2_{ML}(H,A) .$ The map $ Z^2_{ML}(G,A)\to  Z^2_{ML}(H,A)$ given as $(f,h) \mapsto (f_H,h_H) $ is a homomorphism which carries multiplicative Lie 2-coboundaries to multiplicative Lie 2-coboundaries. The induced homomorphism $ H^2_{ML}(G,A)\to  H^2_{ML}(H,A)$ is termed as the \textit{restriction map} and denoted by $Res_{G,H}.$

Now we  define the Bogomolov multiplier ${B_0}(G)$ for a finite multiplicative Lie algebra $G$ utilizing the restriction map as a key tool in the process.

\begin{definition}\label{D2}
For a finite multiplicative Lie algebra $G$, we define the subgroup of $H^2_{ML}(G,\C^*)$ as follows:
	\begin{align*}
		{B_0}(G) =  \bigcap\limits_H \ker(Res_{G,H}) ,
	\end{align*}
	where $H$ is an abelian subalgebra of $G$ and $Res_{G,H}: H^2_{ML}(G,\C^*)\to  H^2_{ML}(H,\C^*).$ We use the term "abelian subalgebra" to refer to an abelian subgroup in which the induced multiplicative Lie algebra structure is trivial. 
\end{definition}
 
In the following proposition, a connection is established between the Bogomolov multiplier and the Schur multiplier of a multiplicative Lie algebra.

\begin{proposition}\label{P2}
Let $G$ be a finite  group with nontrivial multiplicative Lie algebra structure. If each proper subgroup of $G$ is cyclic, then ${B_0}(G) = H^2_{ML}(G,\C^*)\cong \tilde{M}(G)$.
\end{proposition}

\begin{proof}
Given each proper subgroup $H$ of  $G$ is  cyclic, it follows that $H$ is equipped with only trivial multiplicative Lie algebra structure. So, $H$ is  an abelian subalgebra of  $G$. Thus, in each case $ H^2_{ML}(H,\C^*) = 1$. This implies that $\ker(Res_{G,H}) = H^2_{ML}(G,\C^*) $. Hence the result follows.
\end{proof}

\begin{example}
Let $G$ be a group of order $pqr$, where $p,q$, and $r$ are distinct primes with nontrivial multiplicative Lie algebra structure. Then possible orders of proper subgroups $H$ of $G$ are : $1,p,q,r,pq,pr,qr$. If $o(H)\in \{1,p,q,r\}$, $H$ is cyclic. If $o(H)\in \{pq,pr,qr\}$ and $H$ is abelian (non abelian $H$ does not have any role when calculating ${B_0}(G)$, $H$ must be cyclic. Hence
Proposition \ref{P2} implies that   ${B_0}(G) \cong \tilde{M}(G)$.
\end{example}

\begin{example}\label{E1}
Consider the group $V_4 = \langle a, b \mid a^2 = 1 = b^2, ab = ba\rangle$ with multiplicative Lie algebra structure $\star$ given by  $a \star b = a.$ Then Proposition \ref{P2} implies that ${B_0}(V_4) \cong \tilde{M}(V_4)\cong V_4$ by \cite{ADSS}.

On the other hand, if we have an abelian group $G$ with trivial multiplicative Lie algebra structure, then by Definition \ref{D2}, we get ${B_0}(G) = 1$. In particular, ${B_0}(V_4) = 1$.
\end{example}

\begin{example}\label{E2}
	Let $G$ be a finite group with trivial Schur multiplier and a multiplicative Lie algebra structure $\star$ such that $[G, G]$ is a  non abelian  simple group. Then $\tilde{M}(G)$ is trivial by \cite{ADSS}, and thus ${B_0}(G)$ is also trivial.
\end{example}

\begin{definition}\label{D3}
	
	Let $G$  be a multiplicative Lie algebra. Define the subgroup of $G \wedge^L G$ as follows:
	\begin{align*}
		\tilde{M_0}(G) = \langle x\wedge y , [x,y]_0 \mid x\star y = 1 = [x,y]  , x,y\in G  \rangle.
	\end{align*}
	Then $\tilde{M_0}(G) \subseteq \ker(\chi) \cong \tilde{M}(G).$ Consequently, we obtain the factor group denoted as $\tilde B_0(G)$, defined as $\frac{\tilde{M}(G)}{\tilde{M_0}(G)}.$
\end{definition}

\begin{remark}
 Take $V_4$ as in Example \ref{E1} with  $a \star b = a.$ By Definition \ref{Lie}, we get $\tilde{M_0}(V_4) = 1$. As we know $\tilde{M}(V_4) \cong V_4$ by \cite{ADSS}. Hence, we have  $\tilde B_0(V_4) \cong V_4 = {B_0}(V_4)$.
\end{remark}

Now, we have the following result:
\begin{theorem}\label{T2}
Let $G$ be a finite multiplicative Lie algebra. Then ${B_0}(G)$ is isomorphic to $Hom(\tilde B_0(G),\C^*).$ Thus ${B_0}(G) \cong \tilde B_0(G). $
\end{theorem}
\begin{proof}
First, we establish an isomorphism between two mathematical structures $H^2_{ML}(G,\C^*)$ and $Hom(\tilde{M}(G),\C^*)$ in terms of the Lie exterior square.

Consider a pair $(f,h)$ belonging to the set $H^2_{ML}(G,\C^*)$. We begin by constructing a central extension associated with the pair $(f,h)$ (see \cite{RLS}):
\begin{align*}
	1 \longrightarrow \C^* \overset{i}\longrightarrow K  \overset{p}\longrightarrow G \longrightarrow 1.
\end{align*}

Let $t$ be a section of $p.$ We define a map $\Gamma_{(f,h)} : G \wedge^L G \to (K\star K)[K,K] $ sending $[a,b]_0 \mapsto [t(a),t(b)] $ and $c\wedge d \mapsto (t(c)\star t(d)) $, where $a,b,c,d \in G$. This map is independent of chosen section $t.$ Consequently $\Gamma_{(f,h)}$ is a multiplicative Lie algebra homomorphism (and thus a group homomorphism). Additionally, it's worth noting that $\bar{p} \circ \Gamma_{(f,h)} = \chi$, where  $\bar{p} $ is the restriction of $p$ to $(K\star K)[K,K]$ . Now, if we take an element $x$ from $\tilde{M}(G)$, we find that $\Gamma_{(f,h)}(x)$ lies in the kernel of $p$. Consequently, $\Gamma_{(f,h)}(x)$ belongs to $\C^*$. This leads us to conclude that the restriction of $\Gamma_{(f,h)}$ to $\tilde{M}(G)$ (also denoted as $\Gamma_{(f,h)}$) is an element of $Hom(\tilde{M}(G),\C^*)$. It is easy to see that $\Gamma_{(ff',hh')} = \Gamma_{(f,h)}\Gamma_{(f',h')}.$ Thus we have a group homomorphism $\Theta: H^2_{ML}(G,\C^*)\to Hom(\tilde{M}(G),\C^*)$  defined as follows: $(f,h) \mapsto \Gamma_{(f,h)}.$

Conversely, suppose  $\eta \in Hom(\tilde{M}(G),\C^*)$. It is known that every finite multiplicative Lie algebra has at least one covering multiplicative Lie algebra, as stated in \cite{AMS}. Let's denote this covering multiplicative Lie algebra as $G^*$. Then, we can establish a central extension:
\begin{align*}
	1 \longrightarrow N \overset{j}\longrightarrow G^*  \overset{\rho}\longrightarrow G \longrightarrow 1
\end{align*} 
with $N\subseteq (G^*\star G^*)[G^*,G^*]$ and $N \cong \tilde{M}(G).$ Choose a section $\mu:G \to G^*$ of $\rho$ and define maps $f: G\times G \to G^* $ by $f(x,y) = \mu(x)\mu(y)\mu(xy)^{-1}$ and $h: G\times G \to G^* $ by $h(x,y) = (\mu(x)\star \mu(y))\mu(x\star y)^{-1} $ for $x,y \in G.$ It can be readily verified that both $f$ and $h$ map $G\times G$ into $\tilde{M}(G)$, and that the pair $(\eta f, \eta h) \in Z^2_{ML}(G,\C^*)$. The cohomology class of $(\eta f,\eta h)$ does not depend upon the choice of section $\mu.$ By \cite[Proposition 3.6]{RLS}, we have a  homomorphism $\delta: Hom(\tilde{M}(G),\C^*) \to H^2_{ML}(G,\C^*)$ given by $\delta(\eta) = (\eta f,\eta h) + B^2_{ML}(G,\C^*)$. Furthermore, it is the inverse of the map $\Theta$ previously defined.

Now choose $(f,h)\in {B_0}(G) $ and let $\Theta: H^2_{ML}(G,\C^*)\to Hom(\tilde{M}(G),\C^*)$ be defined as above. Denote $\Gamma_{(f,h)} = \Theta(f,h).$ Let $x,y\in G $ and suppose that $x\star y = 1$ and $[x,y] = 1.$ Then $A = \langle x,y \rangle$ is an abelian subalgebra of $G$, therefore $Res_{G,A}(f,h) = 1.$ This implies $\Gamma_{(f,h)}([x,y]_0) = [\mu(x),\mu(y)] = 1$ and $\Gamma_{(f,h)}(x\wedge y) = (\mu(x)\star \mu(y)) = 1.$ So, $	\tilde{M_0}(G)$ is contained in $\ker(\Gamma_{(f,h)}).$ Therefore $\Theta$ induces a homomorphism $\tilde{\Theta}:{B_0}(G) \to Hom(\frac{\tilde{M}(G)}{\tilde{M_0}(G)}, \C^*). $

Let $\eta\in  Hom(\frac{\tilde{M}(G)}{\tilde{M_0}(G)}, \C^*).$ Then we have a homomorphism $ \tilde{\eta}:{\tilde{M}(G)} \to \C^*.$ Put $(f,h) = \delta(\tilde{\eta}).$ Suppose that 
\begin{align*}
	1 \longrightarrow \C^* \overset{i}\longrightarrow K  \overset{p}\longrightarrow G \longrightarrow 1
\end{align*}
is a central extension associated to $(f,h).$ Choose an arbitrary abelian subalgebra $A$ of $G$. Then we have the following central extension 
\begin{align*}
	1 \longrightarrow \C^* \overset{i}\longrightarrow A'  \overset{p|_{A'}}\longrightarrow A \longrightarrow 1
\end{align*}
that corresponds to $Res_{G,A}{(f,h)}.$ Given that for all $a, b \in A$, we have $[a, b] = 1$ and $a\star b = 1$, it implies that $[a, b]_0$ and $a\wedge b$ both belong to ${\tilde{M_0}(G)}$, which is a subset of $\ker(\tilde{\eta})$. Consequently, $[\bar a, \bar b] = 1$ and $\bar a\star \bar b = 1$ within $A'$.
This observation leads to the conclusion that $A'$ is abelian with a trivial multiplicative Lie algebra structure. Hence the corresponding central extension defined above splits. As a result, $Res_{G,A}{(f,h)} = 1 $. Consequently,  $(f, h)$ belongs to ${B_0}(G)$.
Therefore, we can establish a homomorphism, denoted as $\tilde\delta$, induced by $\delta$, from $Hom(\frac{\tilde{M}(G)}{\tilde{M_0}(G)},\C^*)$ to ${B_0}(G)$ and its inverse is $\tilde{\Theta}$.
As a result, we conclude that ${B_0}(G) \cong \tilde B_0(G)$.
\end{proof}

Let $K(G)$ denote the set $\{(x\star y)[x,y]:x,y\in G\}$. In the following theorem we derive a Hopf-type formula for $\tilde{B_0}(G)$.

\begin{proposition}\label{P3}
	Let $G$ be a multiplicative Lie algebra with free presentation $G \cong \frac{F}{R}$. Then $$\tilde{B_0}(G) \cong \frac{R\cap(F\star F)[F,F]}{\langle \ K(F)\cap R \rangle}.$$
\end{proposition}

\begin{proof}
	From  \cite{RLS}, there exists an isomorphism between $G\wedge^L G$ and $\frac{(F\star F)[R,F]}{(R\star F)[R,F]}$. This isomorphism maps $xR\wedge yR$ to $(x\star y){(R\star F)[R,F]}$ and $[aR,bR]_0$ to $[a, b]{(R\star F)[R,F]}$.
	
	Without loss of generality, consider arbitrary elements $xR\wedge yR$ and $[xR,yR]_0$ in $\tilde{M_0}(G)$. This implies that $(x\star y) \in R$ and $[x, y] \in R$. Applying the above isomorphism, we have $(x\star y)[x, y]{(R\star F)[R,F]} \in \langle  \frac{R\cap(F\star F)[F,F]}{(R\star F)[R,F]} \cap  K(\frac{F}{(R\star F)[R,F]})  \rangle =  \langle K(\frac{F}{(R\star F)[R,F]})\cap \frac{R}{(R\star F)[R,F]}\rangle =  \frac{\langle K(F)\cap R\rangle}{(R\star F)[R,F]}.$
	Consequently, we can deduce that $\tilde{M_0}(G) \cong \frac{\langle K(F)\cap R\rangle}{(R\star F)[R,F]}.$ Then, by invoking Theorem \ref{T1}, the desired result follows.
\end{proof}

It's noteworthy that every finite multiplicative Lie algebra has at least one covering multiplicative Lie algebra, as established in \cite{AMS}. Indeed, it's crucial to emphasize that the covering multiplicative Lie algebra, denoted as $G^*$, may not be unique, as elaborated upon in \cite{AMS}. However, what is particularly interesting is that for any choice of $G^*$, the following isomorphism holds: $$(G^*\star G^*) [G^*,G^*]\cong \frac{(F\star F)[F,F]}{(R\star F)[R,F]}.$$ This identification allows us to derive an alternative formulation of $\tilde{B_0}(G)$.
 
\begin{proposition}\label{P4}
	Let $G$ be a multiplicative Lie algebra and $G^*$ its cover. Then 
	$$ \tilde{B_0}(G) \cong \frac {\tilde{M}(G)}{\langle K(G^*)\cap {\tilde{M}(G)} \rangle}.$$
\end{proposition}

\begin{proof}
	Since $G^*$ is a covering of $G,$ we have the following central extension of  multiplicative Lie algebras
	\begin{align*}
	1\longrightarrow N \longrightarrow G^* \overset{\phi}\longrightarrow G\longrightarrow 1
	\end{align*} 
	with $N \cong \tilde M(G)$ and $N \subseteq (G^*\star G^*) [G^*,G^*].$ The proof of Proposition \ref{P2} implies that $\tilde{M_0}(G) \cong \frac{\langle K(G^*)\cap N\rangle}{(N\star G^*)[N,G^*]} \cong {\langle \ K(G^*)\cap {\tilde{M}(G)} \rangle}.$ Hence the result follows.
\end{proof}

In particular, $ \tilde{B_0}(G)$ is trivial if and only if every element of $\tilde{M}(G)$ can be represented as a product of elements of $ K(G^*)$ that all belongs to $\tilde{M}(G).$ 

\begin{theorem}\label{T3}
		Let $G$ be a multiplicative Lie algebra and $H$ be an ideal of $G$. Then we have the following five term exact sequence:
		\begin{align*}
	\tilde{B_0}(G) \longrightarrow \tilde{B_0}(\frac{G}{H}) \longrightarrow  \frac{H}{\langle K(G)\cap H \rangle} \longrightarrow G^{ab}\longrightarrow (\frac{G}{H})^{ab}\longrightarrow 1 
			\end{align*}
where $G^{ab}$ stands for the abelianizer $\frac{G}{(G\star G)[G,G]}$ of the multiplicative Lie algebra $G$.
\end{theorem}	
	\begin{proof}
		Let $G$ have the free presentation 
		\begin{align*}
			1 \longrightarrow R \longrightarrow  F \longrightarrow G\longrightarrow 1 
		\end{align*}
		and let  \\
		\begin{align*} 
			1 \longrightarrow R \longrightarrow RT \longrightarrow H\longrightarrow 1
		\end{align*}
		be the corresponding free presentation of $H$, where $T$ is an ideal of $F$. Then Proposition \ref{P2} implies that $\tilde{B_0}(G)\cong \frac{R\cap(F\star F)[F,F]}{\langle K(F)\cap R \rangle}$ and $\tilde{B_0}(\frac{G}{H})\cong \frac{RT\cap(F\star F)[F,F]}{\langle K(F)\cap RT \rangle}$.  The canonical epimorphism $\nu:G\longrightarrow \frac{G}{H}$ induces a multiplicative Lie algebra homomorphism ${\nu}^*:\tilde{B_0}(G)\longrightarrow \tilde{B_0}(\frac{G}{H})$.
		Then $\ker({\nu}^*)=\frac{R\cap \langle K(F)\cap RT\rangle}{\langle K(F)\cap R\rangle}$ and $\im({\nu}^*)=\frac{(F\star F)[F,F]\cap \langle K(F)\cap RT \rangle R}{\langle K(F)\cap RT\rangle}$. Now it can be seen that $\langle K(G)\cap H \rangle=\langle K(\frac{F}{R})\cap \frac{RT}{R}\rangle\cong \frac{\langle K(F)\cap RT\rangle R}{R}$, therefore $\frac{H}{\langle K(G)\cap H\rangle}\cong \frac{RT}{\langle K(F)\cap RT\rangle R}$.  Thus we  have a natural multiplicative Lie algebra homomorphism $\phi:\tilde{B_0}(\frac{G}{H})\longrightarrow \frac{H}{\langle K(G)\cap H \rangle}$. Then $\ker(\phi)=\frac{(F\star F)[F,F]\cap \langle K(F)\cap RT \rangle R}{\langle K(F)\cap RT\rangle}$ and $\im(\phi)=\frac{(RT\cap (F\star F)[F,F])R}{\langle K(F)\cap RT\rangle R}=\frac{(F\star F)[F,F]R\cap RT}{\langle K(F)\cap RT\rangle R}\cong\frac{\frac{(F\star F)[F,F]R\cap RT}{R}}{\frac{\langle K(F)\cap RT \rangle R}{R}}=\frac{(G\star G)[G,G]\cap H}{\langle K(G)\cap H\rangle}$. Again there is a natural multiplicative Lie algebra homomorphism $f :\frac{H}{\langle K(G)\cap H\rangle}\longrightarrow \frac{G}{(G\star G)[G,G]}$ such that kernel of $f$ is $\im(\phi)$ and $\im(f)=\frac{H(G\star G)[G,G]}{(G\star G)[G,G]}$. Finally, at the end again we can define a multiplicative Lie algebra epimorphism $h:\frac{G}{(G\star G)[G,G]}\longrightarrow \frac{\frac{G}{H}}{(\frac{G}{H}\star \frac{G}{H})[\frac{G}{H},\frac{G}{H}]}$ such that $\ker(h)= \im(f)$. From here our result follows.
	\end{proof}

\begin{proposition}\label{P5}
	Let $G$ be a multiplicative Lie algebra with a free presentation 
\begin{align*}
	1 \longrightarrow R \longrightarrow  F \longrightarrow G\longrightarrow 1 
\end{align*}
and let $H = \frac{RT}{R} $ be an ideal of $G$. Then the sequence  
	
\begin{align*}
			1 \longrightarrow	\frac{R\cap \langle K(F)\cap RT \rangle}{\langle K(F)\cap R \rangle} \longrightarrow \tilde{B_0}(G) \longrightarrow \tilde{B_0}( \frac{G}{H}) \longrightarrow \frac{H\cap {(G\star G)[G,G]}}{\langle K(G)\cap H \rangle}\longrightarrow 1
\end{align*}
	is exact.
\end{proposition}

\begin{proof}
The proof of Theorem \ref{T3} implies that $\ker({\nu}^*:\tilde{B_0}(G)\to \tilde{B_0}(\frac{G}{H}))=\frac{R\cap \langle K(F)\cap RT\rangle}{\langle K(F)\cap R\rangle}$ and $\im(\phi:\tilde{B_0}(\frac{G}{H})\longrightarrow \frac{H}{\langle K(G)\cap H \rangle}) \cong \frac{H\cap {(G\star G)[G,G]}}{\langle K(G)\cap H \rangle}.$
\end{proof}	

\subsection{The curly Lie exterior square}
The Lie exterior square provides essential insights into the Schur multiplier of a multiplicative Lie algebra (see \cite{RLS}). In the following discussion, we introduce a related construction that serves a similar purpose when examining $\tilde{B_0}(G)$.

Let $G$ be a multiplicative Lie algebra. Then the curly Lie exterior square $G \curlywedge^L G $ of $G$ is a multiplicative Lie algebra generated by the set $\{x\curlywedge y : x,y\in G\}\cup \{[x,y]_1 : x,y\in G \}$ of symbols subject to the following relations:
\begin{enumerate}
	\item $x\curlywedge y=1$ if $x\star y=1$, and $[x,y]_1=1$ if $[x,y]=1$,
	\item  $(x\curlywedge y)(y\curlywedge x)=1=[x,y]_1 [y,x]_1$,
	\item $(x\curlywedge yz)^{-1}(x\curlywedge y)(^yx\curlywedge {^yz})=1=[x,yz]_1^{-1}[x,y]_1[^yx,{^yz}]_1$,
	\item $(xy\curlywedge z)^{-1}(^xy\curlywedge {^xz})(x\curlywedge z)=1=[xy,z]_1^{-1}[^xy,{^xz}]_1[x,z]_1$,
	\item $[x\curlywedge y,u\curlywedge v]=[[x,y]_1,u\curlywedge v]=[x,y]_1\bar{\star}(u\wedge v)=[x,y]_1\bar{\star}[u,v]_1$,
	\item $(x\curlywedge y)(^zy\curlywedge {^zx})=(^{xy}x^{-1}\curlywedge {^xz})(x\curlywedge z)$,
	\item $[x,y]_1[^zy,{^zx}]_1=[^{xy}x^{-1},{^xz}]_1[x,z]_1$,
	\item $[[x,y]_1,[u,v]_1]=[[x,y],[u,v]]_1$,
	\item $(x\curlywedge y)\bar{\star}(u\curlywedge y)=(x\star y)\curlywedge (u\star v)$,
\end{enumerate}
for all $ x,u,y,v,z \in G.$  
Observe that the relations satisfied by the symbols $x\curlywedge y$ and $[x,y]_1$ is same as that are in the definition of Lie exterior square of a multiplicative Lie algebra except relation $(1)$. Therefore we  have the following short exact sequence
\begin{align*}
	1 \longrightarrow \tilde{B_0}(G) \longrightarrow  G \curlywedge^L G   \overset{\bar\chi} \longrightarrow (G\star G)[G,G] \longrightarrow 1 ,
\end{align*}
 where ${\bar\chi}(x\curlywedge y)= x\star y$ and ${\bar\chi}([z,w]_1)= [z,w]$.

If $G$ is perfect multiplicative Lie algebra, then  $G \wedge^L G $ is also a perfect multiplicative Lie algebra, and 
\begin{align*}
	1 \longrightarrow \tilde{M}(G) \overset{i} \longrightarrow  G \wedge^L G   \overset{\chi} \longrightarrow G\longrightarrow 1 
\end{align*}
is a universal central extension \cite[Corollary 8.7]{RLS}. A similar description can be obtained for $G \curlywedge^L G .$

\begin{definition} 	
	\begin{enumerate}
	\item  We say that a central extension 
	\begin{align*}
		1 \longrightarrow A  \longrightarrow  H    \overset{\pi} \longrightarrow G\longrightarrow 1 
	\end{align*}
	of a multiplicative Lie algebra $G$ is commutativity-triviality preserving (CTP) if for any $a,b \in G$ satisfying $(a\star b)[a,b] = 1 $, there exists $x,y \in H $ such that $\pi(x) = a , \pi(y) = b $ with $(x\star y )[x,y] = 1 $.
	
	\item A CTP extension 
	\begin{align*}
		1 \longrightarrow A  \longrightarrow  U   \overset{\phi} \longrightarrow G\longrightarrow 1 
	\end{align*}
	of a multiplicative Lie algebra $G$ is said to be universal CTP extension if for every CTP extension 
	\begin{align*}
		1 \longrightarrow A'  \longrightarrow  U'  \overset{\phi'} \longrightarrow G\longrightarrow 1 
	\end{align*}
	of $G$, there exists a multiplicative Lie algebra homomorphism $\psi: U \to U'$ such that $\phi'\psi = \phi .$
	\end{enumerate}

\end{definition}

\begin{lemma}\label{L1}
	If 
	\begin{align*}
		1 \longrightarrow A  \longrightarrow  U   \overset{\phi} \longrightarrow G\longrightarrow 1 
	\end{align*}
	is a universal CTP extension, then $U$ is perfect.
\end{lemma}
\begin{proof}
	Suppose that $U$ is not perfect. Then $\frac{U}{(U\star U)[U,U]}$ is a nontrivial abelian group with trivial multiplicative Lie algebra. It is easy to see that the direct product extension
	\begin{align*}
		1 \longrightarrow \frac{U}{(U\star U)[U,U]}  \overset{i_1} \longrightarrow  \frac{U}{(U\star U)[U,U]}\times G   \overset{p_2} \longrightarrow G\longrightarrow 1 
	\end{align*}
	is a CTP extension  of $G$. The rest of the proof follows from \cite[Proposition 8.4]{RLS} in which there exists two Lie algebra homomorphisms from $U$ to $\frac{U}{(U\star U)[U,U]}$ such that the respective diagram commutes. Hence 
	\begin{align*}
		1 \longrightarrow A  \longrightarrow  U   \overset{\phi} \longrightarrow G\longrightarrow 1 
	\end{align*}
	can not be a CTP-universal extension.
\end{proof}

Since homomorphic image of a perfect multiplicative Lie algebra is a perfect multiplicative Lie algebra, we have the following corollary:
\begin{corollary}
If $G$ admits a universal CTP extension, then $G$ is a perfect multiplicative Lie algebra.
\end{corollary}

\begin{proposition}
If $G$ is a perfect multiplicative Lie algebra , then 
\begin{align*}
E_1:	1 \longrightarrow \tilde{B_0}(G) \longrightarrow  G \curlywedge^L G   \overset{\bar\chi} \longrightarrow G \longrightarrow 1 
\end{align*}
is a universal CTP extension.
\end{proposition}
\begin{proof}
Let $G$ is a perfect multiplicative Lie algebra. Suppose $G$ is given by the free presentation $G = \frac{F}{R}$. Then the following central extension
\begin{align*}
	1 \longrightarrow \frac{R}{\langle K(F)\cap R\rangle} \longrightarrow \frac{F}{\langle K(F)\cap R\rangle}  \overset{\pi} \longrightarrow G
	 \longrightarrow 1 
\end{align*}
 induces the perfect central extension of $G = (G\star G)[G,G]$:
\begin{align*}
E_2:	1 \longrightarrow \frac{R\cap(F\star F)[F,F]}{\langle K(F)\cap R\rangle} \longrightarrow \frac{(F\star F)[F,F]}{\langle K(F)\cap R\rangle}  \overset{\pi} \longrightarrow G \longrightarrow 1 
\end{align*}
 We claim that it is a universal CTP extension. For that let $a,b \in G \cong \frac{(F\star F)[F,F]R}{R} $ satisfying $(a\star b)[a,b] = 1 $. Then there exists $x,y \in (F\star F)[F,F] $ such that $ a = xR , b = yR $. This implies $(x\star y)[x,y] \in K(F)\cap R \subseteq \langle K(F)\cap R\rangle $. This shows that the above central extension of $G$ is CTP. Next, let 
 \begin{align*}
 	1 \longrightarrow A  \longrightarrow  U   \overset{\phi} \longrightarrow G\longrightarrow 1 
 \end{align*}
 be another CTP extension of $G$. As $F$ is free multiplicative Lie algebra, there is a homomorphism  $\tau : F \to U$ such that $\phi \tau = \pi$. Take an arbitrary $(x\star y)[x,y] \in K(F)\cap R$, where $x,y \in F$. Since $(\pi(x)\star \pi(y) )[\pi(x),\pi(y)] = 1 $, there exists $z,w \in U $ such that $(z\star w)[z,w] = 1 $. We can write $\tau(x) = zz'$ and $\tau(y) = ww'$ for some $z',w' \in A $. Then $\tau((x\star y)[x,y]) = 1$, hence $\tau $ induces a homomorphism $\nu: \frac{F}{\langle K(F)\cap R\rangle} \to U$. The restriction of $\nu$ to $ \frac{(F\star F)[F,F]}{\langle K(F)\cap R\rangle}$ gives the required map. Also $E_2$ is equivalent to the extension $E_1$. Hence the result follows.
\end{proof}

\begin{proposition}\label{P6}
A CTP extension 
\begin{align*}
	1 \longrightarrow A  \longrightarrow  U   \overset{\phi} \longrightarrow G\longrightarrow 1 
\end{align*}
of a multiplicative Lie algebra $G$ is universal CTP extension if and only if $U$ is perfect, and every  CTP extension of $U$ splits.
\end{proposition}

\begin{proof}
	First assume that $U$ is perfect  multiplicative Lie algebra, and that every  CTP extension of $U$ splits. Let 
	\begin{align*}
		1 \longrightarrow A'  \longrightarrow  U'  \overset{\phi'} \longrightarrow G\longrightarrow 1 
	\end{align*}
	be an arbitrary CTP extension of $G.$ Consider the Lie subalgebra  $U\times_G U' = \{(u,u')\in U\times U' : \phi(u) = \phi'(u') \}$ of $ U\times U'$  and let $\pi : U\times U' \to U$ be the projection to the first factor. Then 
	\begin{align*}
		1 \longrightarrow \ker(\pi)  \longrightarrow U\times_G U'  \overset{\pi} \longrightarrow U\longrightarrow 1 
	\end{align*}
is a CTP extension of $U.$	From our hypothesis, the sequence splits. Let $t$ be an splitting of $\pi.$ Then there is a homomorphism $\psi: U \to U'$ such that $t(u) = (u, \psi(u)) \in U\times_G U' .$ Since $U$ is perfect, $\psi$ is uniquely determined. Also $\phi'(\psi(u)) = u.$ Thus the extension
\begin{align*}
	1 \longrightarrow A  \longrightarrow  U   \overset{\phi} \longrightarrow G\longrightarrow 1 
\end{align*}
is CTP-universal.

Conversely, suppose that the extension 
\begin{align*}
E:	1 \longrightarrow A  \longrightarrow  U   \overset{\phi} \longrightarrow G\longrightarrow 1 
\end{align*}
is universal CTP extension. Then by the Lemma \ref{L1}, $U$ is perfect. Let 
\begin{align*}
E':	1 \longrightarrow B  \overset{i}\longrightarrow  K   \overset{\psi} \longrightarrow U\longrightarrow 1 
\end{align*}
is a CTP extension of $U$. Then the extension 
\begin{align*}
E'':	1 \longrightarrow \ker(\phi \psi)  \overset{i}\longrightarrow  K   \overset{\phi \psi} \longrightarrow G\longrightarrow 1 
\end{align*}
is a central extension of $G$ by \cite[Proposition 8.8]{RLS}. Since $E$ is universal CTP extension, for any $a,b \in G$ satisfying $(a\star b)[a,b] = 1 $, there exists $x,y\in U $ such that $\phi(x) = a , \phi(y) = b $ with $(x\star y)[x,y] = 1 .$ Again, $E'$ is a CTP extension , there exists $x',y' \in K $ such that $\psi(x') = x , \psi(y') = y $ with $(x'\star y')[x',y'] = 1 .$ This implies that $E''$ is a CTP extension. Thus, there is a unique Lie homomorphism $\phi'$ from $U$ to $K$ such that $\phi \psi \phi' = \phi$. This shows that $\psi \phi'$ and $I_U$ are Lie homomorphism from $U$ to $U$ which induce morphism from $E$ to itself. Since $E$ i suniversal CTP extension, it follows that $\psi \phi' = I_U$. This shows that $E'$ splits.

\end{proof}

\begin{theorem}
	Let $G_1$ and $G_2$ be isoclinic multiplicative Lie algebras. Then $\tilde{B_0}(G_1) \cong \tilde B_0(G_2) $.
\end{theorem}
\begin{proof}
	As $G_1$ and $G_2$ are isoclinic multiplicative Lie algebras, there exist multiplicative Lie algebra isomorphisms $\lambda:\frac{G_1}{\mathcal{Z}(G_1)} \longrightarrow \frac{G_2}{\mathcal{Z}(G_2)}$ and $\mu:\ ^M[G_1, G_1]\longrightarrow \ ^M[G_2, G_2]$ such that the image of $^M[a,b]$ under $\mu$ is compatible with the image of $a\mathcal{Z}(G_1)$ and $b\mathcal{Z}(G_1)$ under $\lambda$. This implies, whenever $\lambda(a\mathcal{Z}(G_1)) = a'\mathcal{Z}(G_2)$ and $\lambda(b\mathcal{Z}(G_1)) = b'\mathcal{Z}(G_2)$ then $\mu[a,b] = [a',b']$ and $\mu(a\star b) = (a'\star b')$.
Thus, we have a multiplicative Lie algebra homomorphism $\eta : G_1 \curlywedge^L G_1  \to G_2 \curlywedge^L G_2  $ given by $\eta(a\curlywedge b) = (a'\curlywedge b') $ and $\eta[a,b]_1 = [a',b']_1$.  Similarly, we define $\delta : G_2 \curlywedge^L G_2 \to G_1 \curlywedge^L G_1 $ via $\lambda^{-1}$, which is the inverse of $\eta $. Thus, $\eta $ is an  isomorphism. In this way, we have the following commutative diagram:
	
	\[
	\xymatrix{
	 1\ar[r] &	\tilde B_0(G_1)\ar[d]_{\bar\eta}\ar[r] & G_1 \curlywedge^L G_1  \ar[r]^{~~~~\kappa_1} \ar[d]_{\eta}  & ^M[G_1, G_1]  \ar[d]_{\mu} \ar[r] & 1 \\		1\ar[r] & \tilde B_0(G_2)\ar[r]& G_2 \curlywedge^L G_2 \ar[r]^{~~~~\kappa_2} & ^M[G_2, G_2] \ar[r] & 1 }
	\] 
	where $\bar \eta$ is the restriction of $\eta$. Hence   $\bar \eta$ is also an isomorphism.
\end{proof}

\noindent{\bf Acknowledgement:}
The first named author sincerely thanks IIIT Allahabad and the University grant commission (UGC), Govt. of India, New Delhi for the research fellowship.

\end{document}